\documentclass[12pt]{article}
\usepackage[a4paper,textwidth=6.3in,textheight=8.7in]{geometry}  % matches e-jc style sheet

\usepackage{amsthm,amsmath,amssymb}
\usepackage{amsfonts}

\usepackage{graphicx}
\usepackage{latexsym}
\usepackage{pifont}
\usepackage{color}
\usepackage[dvipsnames]{xcolor}

\usepackage[colorlinks=true,citecolor=black,linkcolor=black,urlcolor=blue]{hyperref}

\def\T{\textcolor{black}{\textbf{\checkmark}}}
\def\Tt{\textcolor{blue}{\textbf{\checkmark$_2$}}}
\def\Tv{\textcolor{Plum}{\textbf{\checkmark$_7$}}}
\def\Ts{\textcolor{ForestGreen}{\textbf{\checkmark$_8$}}}
\def\F{\textcolor{black}{\textbf{\text{\ding{55}}}}}
\def\Ff{\textcolor{red}{\textbf{\text{\ding{55}$_8$}}}}
\def\Fv{\textcolor{red}{\textbf{\text{\ding{55}$_5$}}}}
\def\Q{\textbf{?}}

\newtheorem{theorem}{Theorem}
\newtheorem{corollary}{Corollary}[theorem]
\newtheorem{question}{Question}

\title{Uniquely $K^{(k)}_r$-saturated Hypergraphs}

\author{
  Andr\'{a}s Gy\'{a}rf\'{a}s\thanks{Research was supported in part by grant (no.\ K116769) from the National Research, Development and Innovation Office---NKFIH.}\\
  \small Alfr\'ed R\'enyi Institute of Mathematics\\[-0.8ex]
  \small Hungarian Academy of Sciences\\[-0.8ex]
  \small Budapest, P.O. Box 127\\[-0.8ex]
  \small Budapest, Hungary, H-1364\\
  \small\tt gyarfas.andras@renyi.mta.hu\\
  \and
  Stephen G. Hartke\thanks{Partly supported by a U.S. Fulbright Scholar Fellowship and by a grant from the Simons Foundation (\#316262 to Stephen Hartke).} \quad Charles Viss\\
  \small Dept.\ of Mathematical and Statistical Sciences\\[-0.8ex]
  \small University of Colorado Denver\\[-0.8ex]
  \small Denver, CO 80217 USA\\
  \small\tt stephen.hartke@ucdenver.edu,\\[-0.8ex]
  \small\tt charles.viss@ucdenver.edu
}

\date{\small Submitted: December 7, 2017 \\ %\dateline{submitted}{accepted}\\
\small 2010 Mathematics Subject Classifications: 
05C65, % hypergraphs
05D15, % Transversal (matching) theory
05D05, % extremal set theory
05B05. % block designs
}

\begin{document}

\maketitle 

\begin{abstract}
In this paper we generalize the concept of uniquely $K_r$-saturated graphs to hypergraphs. Let $K_r^{(k)}$ denote the complete $k$-uniform hypergraph on $r$ vertices. For integers $k,r,n$ such that $2\le k <r<n$, a $k$-uniform hypergraph $H$ with $n$ vertices is  {\em uniquely $K_r^{(k)}$-saturated} if $H$ does not contain $K_r^{(k)}$ but adding to $H$ any $k$-set that is not a hyperedge of $H$ results in {\em exactly one} copy of  $K_r^{(k)}$. Among uniquely $K_r^{(k)}$-saturated hypergraphs, the interesting ones are the {\em primitive} ones that do not have a dominating vertex---a vertex belonging to all possible ${n-1\choose k-1}$ edges. Translating the concept to the complements of these hypergraphs, we obtain a natural restriction of $\tau$-critical hypergraphs: a hypergraph $H$ is {\em uniquely $\tau$-critical} if for every edge $e$, $\tau(H-e)=\tau(H)-1$ and $H-e$ has a unique transversal of size $\tau(H)-1$.

We have two constructions for primitive uniquely $K_r^{(k)}$-saturated hypergraphs. One shows that for $k$ and $r$ where $4\le k<r\le 2k-3$, there exists such a hypergraph for every $n>r$. This is in contrast to the case $k=2$ and $r=3$ where only the Moore graphs of diameter two have this property. Our other construction keeps $n-r$ fixed; in this case we show that for any fixed $k\ge 2$ there can only be finitely many examples.  We give a range for $n$ where these hypergraphs exist. For $n-r=1$ the range is completely determined: $k+1\le n \le {(k+2)^2\over 4}$. For larger values of $n-r$ the upper end of our range reaches approximately half of its upper bound. The lower end depends on the chromatic number of certain Johnson graphs.
\end{abstract}

\section{Introduction}
A graph $G$ is said to be \textit{$H$-saturated} if $G$ contains no subgraph isomorphic to $H$ and if for every edge $e$ in the complement of $G$, the graph $G + e$ does contain a subgraph isomorphic to $H$. A classic result regarding such graphs is provided by Tur\'{a}n~\cite{Turan}: the maximum number of edges in a $K_r$-saturated graph on $n$ vertices is achieved by the nearly-balanced complete $(r-1)$-partite graph on $n$ vertices. Note that in this graph, although $K_r$ does not appear as a subgraph, adding any missing edge results in many copies of $K_r$. 

On the other hand, Erd\H{o}s, Hajnal, and Moon~\cite{Moon} showed that the minimum number edges in a $K_r$-saturated graph on $n$ vertices is achieved by an $(r-2)$-clique joined with an independent set of size $n-r+2$. In this graph, adding any missing edge results in exactly one copy of $K_r$. Motivated by this phenomenon, a graph $G$ is said to be \textit{uniquely $H$-saturated} if $G$ contains no copy of $H$ and if for every edge $e$ in the complement of $G$, the graph $G + e$ contains exactly one copy of $H$. 

Cooper, Lenz, LeSaulnier, Wenger, and West~\cite{Cooper} initiated the study of uniquely saturated graphs by classifying all uniquely $C_k$-saturated graphs for $k \in \{3,4\}$, proving that there are only finitely many such graphs in each case. Later, Wenger and West \cite{Wenger} proved that the uniquely $C_5$-saturated graphs are precisely the friendship graphs and that there are no uniquely $C_k$-saturated graphs for $k \in \{6,7\}$. Furthermore, they prove that there are only finitely many $C_k$-saturated graphs for all $k \geq 6$ and conjecture that no such graphs exist. Other studies of uniquely $H$-saturated graphs include that of Berman, Chappell, Faudree, Gimbel, and Hartman~\cite{Berman}, in which the authors characterize uniquely $T$-saturated graphs for certain trees $T$. 

The study of $H$-saturated graphs most relevant to this paper was undertaken by Hartke and Stolee~\cite{Hartke}, who examined the case $H = K_r$. The uniquely $K_3$-saturated graphs were already characterized in \cite{Cooper} since $K_3 = C_3$, but prior to this study, few examples of these graphs were known for $r > 3$. Using a computational search, Hartke and Stolee discovered several new uniquely $K_r$-saturated graphs for $4 \leq r \leq 7$ and two new infinite families of uniquely $K_r$-saturated Caley graphs. Nevertheless, uniquely $K_r$-saturated graphs appear to quite sporadic. In fact, for all $r \geq 3$, Cooper conjectured that there are only finitely many uniquely $K_r$-saturated graphs that do not contain a dominating vertex.

In this paper, we generalize the concept of uniquely $K_r$-saturated graphs by considering them within the context of hypergraphs. Specifically, for $k < r$, let $K_r^{(k)}$ denote the complete $k$-uniform hypergraph on a set of $r$ vertices. We then say that a $k$-uniform hypergraph $H = (V,E)$ is \textit{uniquely $K_r^{(k)}$-saturated} if $H$ contains no copy of $K_r^{(k)}$ and if for any $k$-set $S$ of $V(H)$ in the complement of $H$, the graph $H + S$ contains exactly one copy of $K_r^{(k)}$. Our primary goal is to determine the values of $r$, $k$, and $n$ for which uniquely $K_{r}^{(k)}$-saturated hypergraphs do or do not exist on $n$ vertices.

In our search for uniquely $K_r^{(k)}$-saturated hypergraphs on $n$ vertices, we ignore trivial examples by assuming $k < r < n$. Furthermore, in a $k$-uniform hypergraph $H$, we say that a vertex $v \in V(H)$ is a \textit{dominating vertex} if all $k$-sets of $V(H)$ that contain $v$ are hyperedges of $H$. Note that this definition generalizes the concept of dominating vertices in graphs. To see the importance of these vertices, suppose that a hypergraph $H$ is uniquely $K_r^{(k)}$-saturated and contains a dominating vertex $v$. It follows that $H - v$ is uniquely $K_{r-1}^{(k)}$-saturated. In other words, $H$ can be formed by simply adding a dominating vertex to a smaller uniquely $K_{r-1}^{(k)}$-saturated hypergraph, and it does not provide any new information regarding the characterization of uniquely $K_{r}^{(k)}$-saturated hypergraphs. Therefore, we seek hypergraphs that do not contain a dominating vertex. Similar to the definition used in \cite{Hartke}, we call such hypergraphs  \textit{primitive} uniquely $K_{r}^{(k)}$-saturated hypergraphs.

One result of this paper is that unlike the case where $k = 2$, many of these hypergraphs exist for uniformities $k \geq 4$. Specifically, we show that for any integers $k,r$ where $4 \leq k < r \leq 2k - 3$, there exists a primitive uniquely $K_{r}^{(k)}$-saturated hypergraph on $n$ vertices for all $n > r$. In order to show this, we construct the \textit{complementary hypergraph} of a uniquely $K_{r}^{(k)}$-saturated hypergraph $H$, which consists of the complements of all $k$-sets of $V(H)$ that are not hyperedges of $H$. These complementary hypergraphs provide insight into the structure and properties of uniquely $K_{r}^{(k)}$-saturated hypergraphs and offer a useful framework for proving their existence or nonexistence under certain conditions. 

Additionally, we describe the relationship between uniquely $K_r^{(k)}$-saturated hypergraphs and $\tau$-critical hypergraphs, where $\tau$ denotes the transversal number of a hypergraph as defined in $\cite{Tuza}$. Specifically, we show that $H$ is a primitive uniquely $K_r^{(k)}$-saturated hypergraph on $n$ vertices if and only if its $k$-uniform complement $H^c$ is $\tau$-critical with $\tau(H^c)=n-r+1$ and for each $S \in E(H^c)$ the hypergraph $H^c - S$ has a unique transversal of size $n-r$. We say that such a hypergraph is \textit{uniquely $\tau$-critical}.

Using this relationship, we see that if the difference $n-r$ between clique size and number of vertices in the hypergraph is bounded, then primitive uniquely $K_r^{(k)}$-saturated hypergraphs do not exist on $n$ vertices for $n$ sufficiently large. We also provide a construction of uniquely $\tau$-cricital hypergraphs to prove a range for $n$ where these hypergraphs do exist. The upper end of this range reaches approximately half of our proven upper bound, and the lower end depends on the chromatic number of certain Johnson graphs. Furthermore, we show that when $n-r=1$, the range for $n$ where these hypergraphs exist is completely determined: $k+2 \leq n \leq \frac{(k+2)^2}{4}$.

Finally, we report the results of a computational search to determine when primitive uniquely $K_r^{(k)}$-saturated hypergraphs exist for relatively small values of $k$, $r$, and $n$. The results of this search suggest that many such hypergraphs exist outside the parameter range of our contructions and offer directions for future research.

\section{Complementary Hypergraphs}
Given a $k$-uniform hypergraph $H$ on $n$ vertices, set $t = n - k$. The \textit{complementary hypergraph} of $H$ is the $t$-uniform hypergraph $R$ where $V(R) = V(H)$ and 
\begin{align*}
E(R) = \left\lbrace S \in \binom{V(H)}{t} : V(H) \setminus S \notin E(H)  \right\rbrace.
\end{align*}
That is, $E(R)$ consists of the complements of all $k$-sets of $V(H)$ that are not hyperedges of $H$. Note that this operation can easily be reversed to obtain the original hypergraph $H$ from $R$, and $H$ is the complementary hypergraph of $R$. Hence, we say that $H$ and $R$ are a pair of complementary hypergraphs. The following theorem characterizes the complementary hypergraphs of uniquely $K_{r}^{(k)}$-saturated hypergraphs and serves as a useful tool in our search for these hypergraphs.

\begin{theorem}\label{thm:complementary_hypergraphs}
Let $H$ be a $k$-uniform hypergraph on $n$ vertices. Set $t = n - k$ and for any $r > k$, set $s = r - k$. Then $H$ is a primitive uniquely $K_{r}^{(k)}$-saturated hypergraph if and only if its complementary hypergraph $R$ satisfies the following three properties:
\begin{enumerate}
\item Every $(t-s)$-set of $V(H)$ is covered by a hyperedge of $R$.
\item Every hyperedge of $R$ contains exactly one $(t-s)$-subset that is covered by no other hyperedge of $R$.
\item No vertex appears in all of $E(R)$.
\end{enumerate}
\end{theorem}

Throughout the following proof of this theorem, we assume that all set complement operations, denoted by $^c$, are performed within the universe $V(H)$. Furthermore, for any $(t-s)$-set $S$ of $V(H)$, define the \textit{codegree} of $S$ in $R$ to be the number of hyperedges in $R$ that contain $S$. Also, for any $k$-set $S$ of $V(H)$ that is not a hyperedge of $H$, we say that an $s$-set $T$ is a \textit{completion} of $S$ in $H$ if all $k$-sets of $S \cup T$ (with the exception of $S$) are hyperedges in $H$. Finally, for the purpose of readability, we will use the term ``edge" interchangeably with ``hyperedge".
\begin{proof}
We first show that $R$ satisfies Property 1 in Theorem~\ref{thm:complementary_hypergraphs} if and only if $H$ contains no copy of $K^{(k)}_{r}$. For the forward direction, assume that $R$ satisfies Property 1 and suppose for the purpose of contradiction that $H$ contains a copy of $K^{(k)}_{r}$. Let $S$ denote the vertex set of this subgraph. It follows that $S^c$ is a subset of $V(H)$ of size $n - r = t - s$, so by assumption, it is covered by a hyperedge of $R$. This edge must be of the form $S^c \cup T$ for some $s$-set $T \subset S$. However, this implies that the complement of this edge, which is the $k$-set $S \cap T^c$, is a non-edge of $H$, contradicting the fact that $S \cap T^c$ is a $k$-set in a copy of $K^{(k)}_{r}$ in $H$.

Conversely, suppose that $R$ does not satisfy Property 1, so that there exists a $(t-s)$-set $W$ of $V(H)$ that is not covered by any edge of $R$. Then $|W^c| = k + s = r$, and all $k$-sets of $W^c$ must be edges in $H$. Thus, $W^c$ induces a copy of $K_r^{(k)}$ in $H$.

Next, we show that $R$ satisfies Property 2 if and only if each non-edge $k$-set in $H$ has a unique completion to a copy of $K^{(k)}_{r}$ in $H$. To do this, given any non-edge $k$-set $S$ of $V(H)$, we establish a bijection between completions of $S$ in $H$ and $(t-s)$-sets of $S^c$ with codegree 1 in $R$. For the forward direction, let $T$ be a completion of $S$ in $H$. Then $(S \cup T)^c$ is a $(t-s)$-set of the edge $S^c$ in $R$. Furthermore, by the definition of completion, any non-edge of $H$ other than $S$ must contain at least one element outside of $S \cup T$. Hence, no other edge of $R$ may contain $(S \cup T)^c$, which implies that $(S \cup T)^c$ is a $(t-s)$-set of $S^c$ with codegree 1 in $R$.

For the reverse direction of the bijection, suppose that $S^c$ contains a $(t-s)$-set $W$ with codegree 1 in $R$. Take $T = S^c \setminus W$, so that $|T| = s$. We claim that $T$ is a completion of $S$ in $H$. To see this, let $S'$ be a $k$-set of $S \cup T$ containing at least one element of $T$; it suffices to show that $S'$ is an edge of $H$. Suppose for the purpose of contradiction that $S'$ is a non-edge. Note that since both $S$ and $T$ are disjoint from $W$, it holds that $S' \cap W = \emptyset$. This implies that $(S')^c$ would be another edge of $R$ that contains $W$, contradicting the fact that $W$ has codegree 1 in $R$. Therefore, $S'$ must indeed be an edge of $H$, which implies that $T$ is a completion of $S$ in $H$. Now that this bijection has been established, it follows immediately that for any non-edge $S$ of $H$, there exists a unique completion of $S$ in $H$ if and only if $S^c$ contains a unique $(t-s)$-set of codegree 1 in $R$.

Finally, we show that $R$ satisfies Property 3 if and only if $H$ contains no dominating vertex. Suppose first that $R$ satisfies property 3, and let $v \in V(H)$. Then there exists an edge $S^c$ of $R$ that does not contain $v$. Thus, the $k$-set $S$ must be a non-edge of $H$ containing $v$, so $v$ cannot be a dominating vertex in $H$. Conversely, suppose that $H$ contains no dominating vertex. Then for any $v \in V(H)$, there exists a non-edge $k$-set $S$ of $H$ that contains $v$. Hence, $v$ does not belong to the edge $S^c$ of $R$. Since $v$ was arbitrary, this implies that no vertex belongs to each edge of $R$.  
\end{proof}

The primary implication of Theorem~\ref{thm:complementary_hypergraphs} is that in order to prove or disprove the existence of a primitive uniquely $K_{r}^{(k)}$-saturated hypergraph for any combination of parameters, it suffices to search for a hypergraph $R$ that satisfies the properties of the theorem. We have found it simpler to search for these complementary hypergraphs rather than searching for uniquely $K_{r}^{(k)}$-saturated hypergraphs directly due to the fact that the properties of the theorem are easily verifiable. 

\section{The Double Star Construction}

We now employ Theorem~\ref{thm:complementary_hypergraphs} to prove the existence of primitive uniquely $K_{r}^{(k)}$-saturated hypergraphs for a wide range of parameters. For example, consider the complementary hypergraphs described in Theorem~\ref{thm:complementary_hypergraphs} when $t = 2$ and $s = 1$. They are simple graphs with no isolated vertices in which each edge has exactly one endpoint of degree 1 and where no vertex is an endpoint of every edge. These graphs are precisely the forests of multiple stars where each star has at least two edges. The smallest such forest is a double star consisting of two copies of $P_3$. By taking the complementary hypergraphs of these forests, we see that for all $n \geq 6$, there exists a primitive uniquely $K_{n-1}^{(n-2)}$-saturated hypergraph on $n$ vertices.

By generalizing the concept of a double star to hypergraphs of uniformity $k \geq 4$, we construct hypergraphs that satisfy the conditions of Theorem~\ref{thm:complementary_hypergraphs} and whose complementary hypergraphs are uniquely $K_r^{(k)}$-saturated for certain values of $k$ and $r$.  In the proof of the following theorem, if $j$ is a positive integer, we use the notation $[j]$ to denote the set $\{1,2,\ldots,j\}$.

\begin{theorem}\label{thm:double_star}
Let $k \geq 4$ be given. Then for any clique size $r$ such that $k< r \leq 2k-3$, there exists a primitive uniquely $K_{r}^{(k)}$-saturated hypergraph on $n$ vertices for every $n > r$.
\end{theorem}

\begin{proof}
Set $t = n - k$ and set $s = r - k$, noting that $r \leq 2k-3$ implies $k \geq s + 3$. It will suffice to find a $t$-uniform hypergraph $R$ on $n$ vertices that satisfies the three properties of Theorem~\ref{thm:complementary_hypergraphs}. To do so, assume $V(R) = [n]$ and consider the following two families $A,B$ of $t$-sets of $[n]$:
\begin{align*}
A &= \{ [s] \cup S : S \text{ is a $(t-s)$-set of $\{s+1,...,n\}$ containing some
 element $i < n-t$} \}, \\
B &= \{ X \cup S : S \text{ is a $(t-s)$-set of $\{n-t,...,n-s\}$ and $X = \{n-s+1,...,n\}$} \}.
\end{align*}
Hence, $B$ consists of all $t$-sets of $\{n-t,...,n\}$ containing the $s$-set $X = \{n-s+1,...,n\}$ as a subset. Note that $A$ resembles a $t$-uniform ``star" with center $[s]$ and that $B$ resembles a star with center $X$. We claim that taking $E(R) = A \cup B$ results in the desired hypergraph $R$.

To prove that $R$ has Property 1, let $W$ be a $(t-s)$-set of $[n]$. If $W$ contains any $i < n - t$, then $W \setminus [s]$ must be a subset of at least one $(t-s)$-set $S$ described in the definition of $A$. Hence, $W$ is covered by the edge $[s] \cup S$ of $R$. On the other hand, if $W$ is a subset of $\{n-t,...,n\}$, then $W \setminus X$ must be a subset of at least one $(t-s)$-set $S$ described in the definition of $B$, implying that $W$ is covered by the edge $X \cup S$ of $R$.

Next, to prove that $R$ has Property 2, first consider an edge $E = [s] \cup S$ from $A$. We claim that $S$ is the unique $(t-s)$-set of $E$ with codegree 1 in $R$. To see this, note that $S$ cannot be contained in any other member of $A$, and since $S$ contains some $i < n-t$, it is not contained in any member of $B$. Hence, $S$ indeed has codegree 1. Now let $S'$ be any other $(t-s)$-set of $E$. It then holds that $S'$ contains some element of $[s]$, so that $|S' \setminus [s]| < t-s$. If $S' \setminus [s]$ still contains some $i < n-t$, then it is possible to construct some edge $E' = [s] \cup T$ of $A$ where $S'\setminus [s] \subset T$ by simply choosing elements of $\{n-t,...,n\}$ to make up $T \setminus S'$ in such a way that $T \neq S$. It would then follow that $S' \subset E'$, or that $S'$ has codegree at least 2. 

However, if $S' \setminus [s]$ contains no element $i < n-t$, then in order for $E' = [s] \cup T$ to belong to $A$, we must make sure to include some element $i$, where  $s < i < n-t$, when constructing the $(t-s)$-set $T$. Furthermore, we must also ensure that $T \neq S$. To do this, note that since $n-t = k \geq s+3$, at least two elements ($s+1$ and $s+2$) belong to this interval. Hence, we can select one of $\{s+1, s+2\}$ to belong to $T$, include $S' \setminus [s]$ in $T$, and select the remaining elements of $T$ from $\{n-t,...,n\}$ in such a way that $T \neq S$. It would then follow that $E' \in A$ and that $S' \subset E'$, or that $S'$ has codegree at least 2 in $R$. Therefore, $S$ must indeed be the unique $(t-s)$-set of $E$ having codegree 1 in $R$.

The other case that must be examined to show that $R$ has Property 2 is an edge of form $E = X \cup S$ from $B$. Again, we claim that $S$ is the unique $(t-s)$-set of $E$ with codegree 1 in $R$. Clearly $S$ is not contained in any other members of $B$, and since all elements of $S$ are at least $n-t$, it is not contained in any members of $A$. Thus, $S$ has codegree 1 in $R$. To show that it is unique, let $S'$ be any other $(t-s)$-set of $E$. Then $S'$ contains some member of $X$, implying that $|S' \setminus X| < t-s$. It is now easy to find another edge $E' = X \cup T$ containing $S'$ by simply letting $T$ contain $S' \setminus X$ and other elements of $\{n-t,...,n-s\}$ in such a way that $T \neq S$. This follows from the fact that $|\{n-t,...,n-s\}| = t-s+1 \geq |S' \setminus X| + 2$. Thus, $S'$ has codegree at least 2 in $R$, which implies that $E$ contains exactly one $(t-s)$-set of codegree 1 in $R$ and completes the proof of Property 2.

Lastly, it remains to be shown that no element of $[n]$ belongs to all of $E(R)$, but this is trivial. To see this, note that no $i < n-t$ belongs to any member of $B$. Furthermore, for any $i \geq n-t$, there clearly exists a member of $A$ that does not contain $i$. 
\end{proof}

One implication of this theorem is that for all $k \geq 4$, there are infinitely many primitive uniquely $K_{k+1}^{(k)}$-saturated hypergraphs (at least one on $n$ vertices for each $n \geq k+2$). This is rather surprising considering that for $k = 2$, the primitive uniquely $K_{k+1}^{(k)}$-saturated hypergraphs are precisely the Moore graphs of diameter two \cite{Cooper}, and there are only finitely many such graphs \cite{Hoffman}.

\section{Uniquely $\tau$-critical Hypergraphs}

In addition to the complementary hypergraphs described in Theorem~\ref{thm:complementary_hypergraphs}, uniquely $K_r^{(k)}$-saturated hypergraphs are also closely related to $\tau$-critical hypergraphs where $\tau(H)$ denotes the transversal number of a hypergraph $H$. A \textit{transversal} of $H$ is a subset of $V(H)$ that intersects each hyperedge of $H$. The \textit{transversal number} $\tau(H)$ of $H$ is the minimum size of a transversal of $H$. Furthermore, $H$ is said to be $\tau$-\textit{critical} if the removal of any hyperedge from $H$ decreases its transversal number. Based on these definitions, we say that $H$ is \textit{uniquely $\tau$-critical} if $H$ is $\tau$-critical and for any $S \in E(H)$, the hypergraph $H - S$ contains a unique minimum transversal (which necessarily has size $\tau(H) - 1$). 

Uniquely $\tau$-critical hypergraphs are intimately related with uniquely $K_r^{(k)}$-saturated hypergraphs, as shown in the following theorem.

\begin{theorem}\label{thm:tau_critical}
Let $H$ be a $k$-uniform hypergraph on $n$ vertices and let $H^c$ denote the $k$-uniform hypergraph with vertex set $V(H)$ and edge set $\binom{V(H)}{k} \setminus E(H)$. Then $H$ is a primitive uniquely $K_r^{(k)}$-saturated hypergraph if and only if $H^c$ is a uniquely $\tau$-critical hypergraph with no isolated vertices where $\tau(H^c) = n - r + 1$. 
\end{theorem}

\begin{proof}
Note that a set $S \subseteq V(H)$ is a transversal of $H^c$ if and only if $V(H) \setminus S$ induces a clique in $H$. Hence, $H$ being $K_r$-free is equivalent to $H^c$ having no transversal of size $n - r$. It also follows that for any $S \in \binom{V(H)}{k} \setminus E(H)$, the hypergraph $H + S$ contains a (unique) copy of $K_r$ if and only if its complement $H^c - S$ contains a (unique) transversal of size $n - r$. Finally, note that $H$ has no dominating vertex if and only if each vertex of $V(H)$ is covered by some edge of $H^c$.
\end{proof}

Although the concept of uniquely $\tau$-critical hypergraphs is new, $\tau$-critical hypergraphs were first introduced by Erd\H{o}s and Gallai in 1961 \cite{Gallai}. Various results have been proven regarding $\tau$-critical hypergraphs, including bounds on the size of their vertex set. These bounds also apply to clique-saturated hypergraphs, and hence, uniquely $K_r^{(k)}$-saturated hypergraphs. The strongest known upper bound is provided by Tuza:

\begin{theorem}[Tuza \cite{Tuza}, 1989]\label{thm:tuza}
The maximum number of vertices in a $k$-uniform, $\tau$-critical hypergraph with no isolated vertices and with $\tau(H) = \tau$ is less than
\begin{align*}
\binom{k + \tau - 1}{k-1} + \binom{k + \tau - 2}{k-1}.
\end{align*}
\end{theorem}

\noindent In fact, this bound is best possible apart from a constant factor. 

When translating this result to $K^{(k)}_r$-saturated hypergraphs on $n$ vertices, whose associated $\tau$-critical hypergraphs have transversal number $n - r + 1$, we see that the bound is a function of $k$ and $\ell := n - r$. Fixing $\ell$ and $k$, we obtain an upper bound on $n$. 

\begin{theorem}\label{thm:upper_bound}
Fix $k \geq 2$. Then if the difference between the number of vertices $n$ and the clique size $r$ is fixed at some constant $\ell \geq 1$, there exist no primitive (uniquely) $K^{(k)}_{r}$-saturated hypergraphs on $n$ vertices for any $n$ at least
\begin{align*}
\binom{k + \ell}{k-1} + \binom{k + \ell - 1}{k-1}.
\end{align*}
\end{theorem}

\begin{proof}
By Theorem~\ref{thm:tau_critical}, the complement $H^c$ of a primitive $K_{n - \ell}^{(k)}$-saturated hypergraph $H$ is a $\tau$-critical hypergraph for $\tau = n - (n - \ell) + 1 = \ell + 1$. The bound now follows from Theorem~\ref{thm:tuza}.
\end{proof}

Thus, we see that when $n - r$ is fixed for any uniformity $k \geq 2$, only finitely many primitive uniquely $K_r^{(k)}$-saturated hypergraphs exist. This is a contrast to Theorem~\ref{thm:double_star}, which states that as long as the clique size is not too large, infinitely many primitive uniquely $K_r^{(k)}$-saturated hypergraphs exist for all uniformities $k \geq 4$. 

A natural question to be asked is that of the sharpness of Theorem~\ref{thm:upper_bound}, or if a sharper upper bound can be obtained. For general $\ell \geq 1$, we offer a construction to prove a range for $n$ where primitive uniquely $K_{n-\ell}^{(k)}$-saturated hypergraphs exist. The upper end of this range is approximately half of the bound given by Theorem~\ref{thm:upper_bound}, and the lower end depends on the chromatic number of certain Johnson graphs.

The well known Johnson graph $J(m,k)$ has vertex set ${[m]\choose k}$ with two vertices adjacent if and only if the corresponding $k$-sets intersect in $k-1$ vertices.  Let $\chi(m,k)$ denote the chromatic number of $J(m,k)$; this parameter is well investigated and relates to constant weight codes. Graham and Sloane \cite{gs-80} proved that $\chi(m,k)\le m$ for all $0<k\le m$. Obviously $\chi(m,2)$ is the chromatic index of the complete graph $K_m$, equal to $m-1$ or $m$ according to $m \equiv 0$ or $m\equiv 1$ $\pmod 2$.  The construction in the proof of the following theorem provides uniquely $\tau$-critical hypergraphs for the stated range.

\begin{theorem}\label{thm:tau_construction} Assume that $k\ge 3$, $\ell\ge 1$, and $n$ are integers that satisfy the following inequality:  $$\chi(\ell+k-1,k-1)+\ell+k-1\le n \le {\ell+k-1\choose k-1}+\ell+k-1.$$
 Then there is a uniquely $\tau$-critical $k$-uniform hypergraph $H$ with $n$ vertices and with $\tau(H)=\ell+1$.
\end{theorem}

\begin{proof} Set $u=\chi(\ell+k-1,k-1)$ and $w=n-u-(\ell+k-1)$. Let $H$ be a  hypergraph whose vertex set is $A\cup B$ for two disjoint sets $A,B$ where $|A|=\ell+k-1$ and $|B|=u+w$. Observe that $|A|+|B|=n$ and thus $H$ has $n$ vertices. From the definition of $u$, the $(k-1)$-element subsets of $A$ can be partitioned into $u$ classes $A_1,\dots,A_u$ so that  in every $A_i$ no two sets intersect in $k-2$ elements. We can obviously refine this partition into a partition $A'_1,\dots,A'_{u+w}$, keeping the property that in every $A'_i$ no two sets intersect in $k-2$ elements and all $A'_i$-s are nonempty. Assume $B=\{b_1,\dots,b_{u+w}\}$ and let $E_i$ be the family of sets obtained by adding $b_i$ to all $(k-1)$-element sets of $A'_i$. The edge set of $H$ is then defined to be $\cup_{i=1}^{u+w} E_i$.

Clearly $H$ is a $k$-uniform hypergraph on $n$ vertices. Since any $\ell+1$ vertices of $A$ is a transversal of $H$, $\tau(H)\le \ell+1$. On the other hand, for any edge $e$ of $H$, the set $A\setminus e$ is an $\ell$-element transversal of $H-e$.  We show that this is in fact the only $\ell$-element transversal of $H-e$, which completes the proof.

Assume that $T$ is a transversal of $H-e$, $|T|=\ell$. Set $$\hat{B}=T\cap B,\  \hat{A}=A\setminus (T\cap A).$$ Assume that $|\hat{B}|=h>0$; then $|\hat{A}|=k-1+h$. For any $b\in \hat{B}$ let $x_b$ be the number of edges of $H-e$ containing $b$ and intersecting $\hat{A}$ in $k-1$ elements. Since no two edges containing $b$ intersect in $k-2$ elements of $A$, $|x_b|\le {k+h-1\choose k-2}/(k-1)$. Thus at most $h{k+h-1\choose k-2}/(k-1)$ edges of $H$ have a vertex in $\hat{B}$ and all other vertices in $\hat{A}$. Since all $(k-1)$-element subsets of $\hat{A}$ are in a unique edge of $H$,
the number of edges of $H$ that are disjoint from $T$ is at least $${k+h-1\choose k-1}- h{k+h-1\choose k-2}/(k-1)={k+h-1\choose k-1}{1\over h+1})\ge {h+2\choose 2}{1\over h+1}>1$$
since $k\ge 3$ and $h>0$. This is a contradiction because only the edge $e$ can be disjoint from $T$. Thus $h=0$, so $T\subset A$. No $\ell$-set of $A$ can be a transversal of $H$, so it most hold that $T\cap e=\emptyset$. Therefore, we have $T=A\setminus e$.
\end{proof}

In terms of uniquely saturated hypergraphs, Theorems~\ref{thm:tau_critical} and \ref{thm:tau_construction} imply the following.

\begin{theorem}\label{thm:tau_cosntruction_hypergraphs}
Let $k,r,n$ be integers where $3 \leq k < r < n$ that satisfy the following inequality:  $$\chi(n-r+k-1,k-1)+n-r+k-1\le n \le {n-r+k-1\choose k-1}+n-r+k-1.$$
Then there exists a primitive uniquely $K_r^{(k)}$-saturated hypergraph on $n$ vertices.
\end{theorem} 

Using that  $\chi(m,k)\le m$ \cite{gs-80} and that $\chi(m,2)$ is the chromatic index of $K_m$, Theorem~\ref{thm:tau_cosntruction_hypergraphs} yields the following corollaries.

\begin{corollary}\label{cor:1}
Let $k,r,n$ be integers where $3 \leq k < r < n$ that satisfy the following inequality:  $$2(n-r+k-1)\le n \le {n-r+k-1\choose k-1}+n-r+k-1.$$
Then there exists a primitive uniquely $K_r^{(k)}$-saturated hypergraph on $n$ vertices.
 \end{corollary}

\begin{corollary}\label{cor:2}
Let $r,n$ be integers where $3 < r < n$ that satisfy the following inequality:   $${n-r+3\choose 2}\ge n \ge \left\{ \begin{array}{ll}
   2(n-r)+4 &\mbox{if $n-r$ is odd}\\
   2(n-r)+3 &\mbox{if $n-r$ is even}\end{array}\right.
 $$
Then there exists a primitive uniquely $K_r^{(3)}$-saturated hypergraph on $n$ vertices.
\end{corollary}

Thus, uniquely $\tau$-critical hypergraphs can be used to prove the existence of primitive uniquely $K_r^{(k)}$-saturated hypergraphs on $n$ vertices for a wide range of parameter combinations. In the special case where $\ell = n - r = 1$, we show that the range of parameters is completely determined.

\begin{theorem}\label{thm:n-r=1}
Let $k \geq 4$ be given. There exists a primitive uniquely $K^{(k)}_{n-1}$-saturated hypergraph on $n$ vertices if and only if $k+2 \leq n \leq \frac{(k+2)^2}{4}$.
\end{theorem}
\begin{proof}
We first prove necessity. The lower bound is trivial. For the upper bound, take $n > \frac{(k+2)^2}{4}$, $s = (n-1)-k$, and $t = n-k$, noting that $t-s = 1$. It suffices to show that there does not exist a hypergraph $R$ that satisfies the properties of Theorem~\ref{thm:complementary_hypergraphs} with parameters $n$, $t$, and $s$.

Suppose for the purpose of contradiction that such a hypergraph $R$ does exist and let $x = |E(R)|$. Since $t-s =1$, it follows that each element of $V(R)$ is covered by an edge of $R$, and each edge of $R$ contains exactly one element of $V(R)$ of codegree 1. Let $S$ denote the set containing the elements of $V(R)$ that have codegree 1 in $R$ (so that $|S| = x$), and put $T = [n] \setminus S$. Then each edge of $R$ consists of one element of $S$ and $(t-1)$ elements of $T$, implying that each edge of $R$ misses precisely $|T| - (t-1) = n - x - t + 1$ elements of $T$. Since each element of $T$ must be missed by at least one of the $x$ edges of $R$, it must hold that $x(n-x-t+1) \geq |T| = n-x$. Substituting $k = n-t$ and rearranging this inequality yields $-x^2 + (k+2)x - n \geq 0$.

Note that the quadratic function $f(x) = -x^2 + (k+2)x - n$ is concave, so it takes on nonnegative values if and only if $f$ has at least one real zero. However, the discriminant of the quadratic equation $-x^2 + (k+2)x - n = 0$ is $(k+2)^2 - 4n < 0$. Therefore, it does not hold that $f(x) \geq 0$ for any value of $x$, a contradiction.

\vspace{.1in}

We prove sufficiency by constructing the complements of the desired hypergraphs. Theorem~\ref{thm:tau_critical} implies that these hypergraphs must be uniquely $\tau$-critical for $\tau = 2$. Hence, the removal of any edge should result in a unique minimum transversal consisting of a single vertex.

Let $V$ be the set $[n]$, and for some integer $m$ such that $4 \leq m \leq k$, let $A$ denote the subset $[m]$ of $V$. Define a $k$-uniform hypergraph $H$ with vertex set $V$ and with edges $e_1,...,e_m$ so that each $e_i$ contains $A \setminus \{i\}$ and the parts of the $e_i$'s outside of $A$ define a $(k - m + 1)$-uniform hypergraph $H'$ on the vertex set $V \setminus A$ (with repeated edges allowed) with no isolated vertices. If each vertex in $H'$ belongs to at most $m-2$ edges, a minimum transversal of $H$ has size at least 2, and for any $e_i \in E(H)$, the hypergraph $H - e_i$ has the unique minimum transversal $\{i\}$. 

Thus, it suffices to show that such a hypergraph $H$ exists whenever $k+2 \leq n \leq \frac{(k+2)^2}{4}$. Note that for a fixed value of $m$, the largest possible size of $V$ is $m + m(k - m + 1) = m(k-m+2)$, which occurs when the edges of $H'$ are disjoint. Therefore, if we fix $m = k$, the largest such construction occurs on $2k$ vertices. Since $m = k \geq 4$, we can obtain similar constructions on $n$ vertices for all $n \in \{k+2,...,2k\}$. 

Next, fix $m = \left \lceil \frac{k+2}{2} \right \rceil$. The largest possible construction again occurs on $m(k-m+2) = \left \lceil \frac{k+2}{2} \right \rceil \cdot \left \lfloor \frac{k+2}{2} \right \rfloor = \left \lfloor \left( \frac{k+2}{2} \right)^2 \right \rfloor$ vertices. We can remove vertices from this construction and adjust the edges of $H'$ in order to obtain valid constructions for all $n$ where $2k + 1 \leq n \leq \left \lfloor \left( \frac{k+2}{2} \right)^2 \right \rfloor$. Specifically, always let $e_1 \setminus A = \{m+1,...,k+1\}$ and $e_2 \setminus A = \{k+2,...,2k-m + 2\}$. Select the remaining edges such that $e_i \setminus A \subseteq \{2k-m + 3,...,n\}$ for each $e_i$ and all of $\{2k-m + 3,...,n\}$ is eventually covered, which is possible as long as $n \geq m + 3(k-m+1)$. In the smallest case where $n = m + 3(k-m + 1)$, we have $n = \left \lceil \frac{k+2}{2} \right \rceil + 3\left( \lfloor \frac{k+2}{2}  \rfloor - 1 \right) \leq 2k + 1$. Hence, there exist valid constructions for all $n \in \{k+2,..., \left \lfloor \left( \frac{k+2}{2} \right)^2 \right \rfloor \}$.
\end{proof}

Note that the proof also holds for $k=3$ with the exception of the case $n = 5$, since in the construction where $m = k$, two vertices in $V \setminus A$ is not enough to form a graph with the desired degree requirements.

\section{Computational Search}
A primitive uniquely $K_r^{(k)}$-saturated hypergraph exists on $n$ vertices if and only if there exists a hypergraph $R$ on $n$ vertices satisfying the three properties of Theorem~\ref{thm:complementary_hypergraphs} with parameters $t = n-k$ and $s = r - k$. Hence, a straightforward approach to determine when primitive uniquely $K_r^{(k)}$-saturated hypergraphs exist is to computationally determine for which values of $n$, $t$, and $s$ valid complementary hypergraphs exist.

To this end, we developed an integer program with parameters $n$, $t$, and $s$ that has a feasible solution if and only if a desired complementary hypergraph $R$ exists. Specifically, for each $t$-set $T$ of $[n]$, we introduce a binary variable $x_T$ that indicates whether or not $T$ is included in $E(R)$. Furthermore, for each $(t-s)$-set $S$ of $[n]$, we introduce a binary variable $y_S$ that indicates whether or not $S$ has codegree 1 in $R$. The number of binary variables in this integer program increases exponentially with $n$, so it is only computationally feasible to use this method to determine the existence of relatively small complementary hypergraphs. Nevertheless, this search supports our results and offers directions for future research. The following tables summarize the main results of our search when put into the context of primitive uniquely $K_r^{(k)}$-saturated hypergraphs on $n$ vertices, separated according to uniformity. The symbols $\T$ and $\F$ indicate the existence or nonexistence, respectively, of a desired hypergraph with parameters $n$, $k$, and $r$, and a question mark indicates that we have yet to obtain results for the corresponding parameter combination.

\newcommand{\ar}[1]{\raisebox{-2pt}{$\xrightarrow{\hspace*{2pt}#1\hspace*{2pt}}$}}

\begin{table}[h]
\begin{center}
\setlength\tabcolsep{2pt}
\begin{tabular}{| c || c | c | c | c | c | c | c | c | c |}
\hline
$\mathbf{k=2}$ & $r=3$ & $r=4$ & $r=5$  & $r=6$ & $r=7$ & $r=8$ & $r=9$ & $r=10$ & \\
\hline
\hline
$n=r+1$ & \F & \Fv & \Fv & \Fv & \Fv & \Fv & \Fv & \Fv & \multicolumn{1}{|l|}{\ar{\Fv}} \\
\hline
$n=r+2$ & \T & \F & \Fv & \Fv & \Fv & \Fv & \Fv & \Fv & \multicolumn{1}{|l|}{\ar{\Fv}} \\
\hline
$n=r+3$ & \F & \T & \F & \Fv & \Fv & \Fv & \Fv & \Fv & \multicolumn{1}{|l|}{\ar{\Fv}} \\
\hline
$n=r+4$ & \F & \F & \T & \F & \Fv & \Fv & \Fv & \Fv & \multicolumn{1}{|l|}{\ar{\Fv}} \\
\hline
$n=r+5$ & \F & \F & \F & \T & \F & \Fv & \Fv & \Fv & \multicolumn{1}{|l|}{\ar{\Fv}} \\
\hline
$n=r+6$ & \F & \T & \F & \F & \T & \F & \Fv & \Fv & \multicolumn{1}{|l|}{\ar{\Fv}} \\
\hline
$n=r+7$ & \T & \F & \F & \F & \F & \T & \F & \Fv & \multicolumn{1}{|l|}{\ar{\Fv}} \\
\hline
$n=r+8$ & \F & \T & \F & \F & \F & \F & \T & \F & $ 11$\ar{\Fv} \\
\hline
$n=r+9$ & \F & \T & \F & \T & \F & \F & \F & \T & $ 12$\ar{\Fv} \\
\hline
\end{tabular}
\caption{Existence of primitive uniquely $K_r^{(k)}$-saturated hypergraphs for $k = 2$.}\label{tbl:k=2}
\end{center}
\end{table}

\begin{table}[p]
\begin{center}
\setlength\tabcolsep{2pt}
\begin{tabular}{| c || c | c | c | c | c | c | c | c | c | c | c | c | c |}
\hline
$\mathbf{k=3}$ & \small{$r=4$} & \small{$r=5$} & \small{$r=6$}  & \small{$r=7$} & \small{$r=8$} & \small{$r=9$} & \small{$r=10$} & \small{$r=11$} & \small{$r=12$} & \small{$r=13$} & \small{$r=14$} & \multicolumn{2}{|c|}{} \\
\hline
\hline
$n=r+1$ & \F & \Ts & \Ff & \Ff & \Ff & \Ff & \Ff & \Ff & \Ff & \Ff & \Ff & \multicolumn{2}{|l|}{\ar{\Ff}} \\
\hline
$n=r+2$ & \F & \Tv & \Tv & \Tv & \Tv & \F & \F & \Q & \Q & \Q & \Fv & \multicolumn{2}{|l|}{\ar{\Fv}}\\
\hline
$n=r+3$ & \T & \T & \T & \Tv & \Tv & \Tv & \Tv & \Tv & \Tv & \Q & \Q & \Q & $22$\ar{\Fv}\\
\hline
$n=r+4$ & \F & \T & \T & \Tv & \Tv & \Tv & \Tv & \Tv & \Tv & \Tv & \Tv & \ar{\Tv}$17$ & $32$\ar{\Fv}\\
\hline
$n=r+5$ & \F & \T & \Q & \T & \Q & \Tv & \Tv & \Tv & \Tv & \Tv & \Tv & \ar{\Tv}$23$ & $44$\ar{\Fv}\\
\hline

\end{tabular}
\caption{Existence of primitive uniquely $K_r^{(k)}$-saturated hypergraphs for $k = 3$.}\label{tbl:k=3}
\end{center}
\end{table}

\begin{table}[p]
\begin{center}
\setlength\tabcolsep{2pt}
\begin{tabular}{| c || c | c | c | c | c | c | c | c | c | c | c | c |}
\hline
$\mathbf{k=4}$ & \small{$r=5$} & \small{$r=6$} & \small{$r=7$}  & \small{$r=8$} & \small{$r=9$} & \small{$r=10$} & \small{$r=11$} & \small{$r=12$} & \small{$r=13$} & \small{$r=14$} & \multicolumn{2}{|c|}{} \\
\hline
\hline
$n=r+1$ & \Tt & \Ts & \Ts & \Ts & \Ff & \Ff & \Ff & \Ff & \Ff & \Ff & \multicolumn{2}{|l|}{\ar{\Ff}} \\
\hline
$n=r+2$ & \Tt & \T & \T & \Tv & \Tv & \Tv & \Tv & \Tv & \Tv & \Q & \Q & $ 28$\ar{\Fv} \\
\hline
$n=r+3$ & \Tt & \T & \T & \T & \Tv & \Tv & \Tv & \Tv & \Tv & \Tv & \ar{\Tv}$ 23$ & $ 52$\ar{\Fv} \\
\hline
$n=r+4$ & \Tt & \T & \T & \T & \Tv & \Tv & \Tv & \Tv & \Tv & \Tv & \ar{\Tv}$ 38$ & $ 87$\ar{\Fv} \\
\hline
$n=r+5$ & \Tt & \T & \Q & \T & \Q & \Tv & \Tv & \Tv & \Tv & \Tv & \ar{\Tv}$ 59$ & $135$\ar{\Fv} \\
\hline

\end{tabular}
\caption{Existence of primitive uniquely $K_r^{(k)}$-saturated hypergraphs for $k = 4$.}\label{tbl:k=4}
\end{center}
\end{table}

\begin{table}[p]
\begin{center}
\setlength\tabcolsep{2pt}
\begin{tabular}{|  c || c | c | c | c | c | c | c | c | c | c | c |}
\hline
$\mathbf{k=5}$ & \small{$r=6$} & \small{$r=7$}  & \small{$r=8$} & \small{$r=9$}  & \small{$r=10$} & \small{$r=11$} & \small{$r=12$} & \small{$r=13$} & \small{$r=14$} & \multicolumn{2}{|c|}{} \\
\hline
\hline
$n=r+1$ & \Tt & \Tt & \Ts & \Ts & \Ts & \Ts & \Ff & \Ff & \Ff & \multicolumn{2}{|l|}{\ar{\Ff}} \\
\hline
$n=r+2$ & \Tt & \Tt & \T & \T & \Tv & \Tv & \Tv & \Tv & \Tv & \ar{\Tv}$ 19$ & $ 48$\ar{\Fv} \\
\hline
$n=r+3$ & \Tt & \Tt & \T & \T & \T & \Tv & \Tv & \Tv & \Tv & \ar{\Tv}$ 39$ & $102$\ar{\Fv} \\
\hline
$n=r+4$ & \Tt & \Tt & \T & \T & \Tv & \Tv & \Tv & \Tv & \Tv & \ar{\Tv}$ 74$ & $192$\ar{\Fv} \\
\hline
$n=r+5$ & \Tt & \Tt & \Q & \Q & \Q & \Q & \Tv & \Tv & \Tv & \ar{\Tv}$130$ & $331$\ar{\Fv} \\
\hline

\end{tabular}
\caption{Existence of primitive uniquely $K_r^{(k)}$-saturated hypergraphs for $k = 5$.}\label{tbl:k=5}
\end{center}
\end{table}

\begin{table}[p]
\begin{center}
\setlength\tabcolsep{2pt}
\begin{tabular}{|  c || c | c | c | c | c | c | c | c | c | c | c | c | }
\hline
$\mathbf{k=6}$  & \small{$r=7$}  & \small{$r=8$} & \small{$r=9$} & \small{$r=10$} & \small{$r=11$} & \small{$r=12$} & \small{$r=13$} & \small{$r=14$} & \small{$r=15$} & \small{$r=16$} & \multicolumn{2}{|c|}{} \\
\hline
\hline
$n=r+1$ & \Tt & \Tt & \Tt & \Ts & \Ts & \Ts & \Ts & \Ts & \Ts & \Ff & \multicolumn{2}{|l|}{\ar{\Ff}} \\
\hline
$n=r+2$ & \Tt & \Tt &  \Tt & \T & \T & \Tv & \Tv & \Tv & \Tv & \Tv  & \ar{\Tv}$ 26$ & $ 75$\ar{\Fv} \\
\hline
$n=r+3$ & \Tt &  \Tt &  \Tt & \T & \T & \T & \Tv & \Tv & \Tv & \Tv & \ar{\Tv}$ 61$ & $179$\ar{\Fv} \\
\hline
$n=r+4$ & \Tt & \Tt & \Tt & \T & \T & \Q & \Q & \Tv & \Tv & \Tv & \ar{\Tv}$131$ & $374$\ar{\Fv} \\
\hline
$n=r+5$ & \Tt & \Tt & \Tt & \Q & \Q & \Q & \Tv & \Tv & \Tv & \Tv & \ar{\Tv}$257$ & $709$\ar{\Fv} \\
\hline

\end{tabular}
\caption{Existence of primitive uniquely $K_r^{(k)}$-saturated hypergraphs for $k = 6$.}\label{tbl:k=6}
\end{center}
\end{table}

In Table~\ref{tbl:k=2}, we see that primitive uniquely $K_r$-saturated graphs are fairly uncommon. The diagonal of $\T$s beginning at the entry $(n=r+2, \ r=3)$ corresponds to the complements of odd cycles, which were shown to be uniquely clique-saturated by David Collins and Bill Kay in 2011. The Petersen graph, a uniquely $K_3$-saturated Moore graph, corresponds to the entry $(n=r+7, \ r=3)$, and the entries $n=r+6$ and $n=r+8$ in the column $r=4$ correspond to two sporadic uniquely $K_4$-saturated graphs discovered by Cooper and Collins. 
The $r=4$ and $r=6$ entries in the $n=r+9$ row correspond to uniquely $K_r$-saturated graphs discovered by Hartke and Stolee~\cite{Hartke}.
Other than these, few primitive uniquely $K_r$-saturated graphs exist on $n$ vertices for small values of $n$.
We were not able to solve our integer program for larger values of $n$ and $r$, so Table~\ref{tbl:k=2} includes the results of the computational study included in \cite{Hartke}.

On the other hand, in Tables \ref{tbl:k=3} through \ref{tbl:k=6}, we see that many primitive uniquely $K_r^{(k)}$-saturated hypergraphs exist with uniformity at least 3. Note that the $i$th column of each table corresponds to the situation $s = r-k = i$ and the $j$th row corresponds to $\ell = n - r = j$. In terms of these tables, Theorem~\ref{thm:double_star} implies that certain columns correspond to situations in which the desired hypergraphs will always exist. Specifically, for each $k \geq 4$, any parameter combination within the first $k-3$ columns of the table for uniformity $k$ yields a desired hypergraph. The entries for which Theorem~\ref{thm:double_star} applies are each marked with a \text{\Tt}. 

On the other hand, Theorem~\ref{thm:upper_bound} implies that in any row of these tables, entries far enough to the right correspond to parameter combinations for which no desired hypergraphs exist. This phenomenon is visible in the second row of Table~\ref{tbl:k=3}; an entry beyond the upper bound is marked with \Fv. Theorem~\ref{thm:tau_cosntruction_hypergraphs} implies that a large portion of the entries in each row before approximately half of this upper bound indeed correspond to situations in which desired hypergraphs exist. Those entries covered by Theorem~\ref{thm:tau_cosntruction_hypergraphs} as implied by Corollaries \ref{cor:1} and \ref{cor:2} or bounds on $\chi(n-r+k-1, k-1)$ given by \cite{eb-96} are marked with \text{\Tv}.

We also observe the sharpness of Theorem~\ref{thm:n-r=1}; those entries corresponding to situations in which the hypergraphs do not exist as described by the theorem are each marked with \text{\Ff}, and those that can be constructed using the method in the proof of the theorem are marked with a \text{\Ts} if the parameter combination is not covered by Theorem~\ref{thm:double_star}.

\section{Conclusions and Future Directions}

In this paper, we have shown the existence and nonexistence of various primitive uniquely $K_r^{(k)}$-saturated hypergraphs by examining their related complementary hypergraphs and $\tau$-critical hypergraphs. Although few of these graphs are known to exist when $k=2$, we have shown that many of these hypergraphs indeed exist for larger uniformities via special constructions. However, we also see that many parameter combinations do not permit these hypergraphs, leaving more questions to be answered regarding the precise parameter combinations for which these hypergraphs exist. 

Interesting phenomena encountered during our computational search include the results for $k = 3$, $r=4$, which can be found in the first column of Table~\ref{tbl:k=3}. This table shares a similar structure with Tables~\ref{tbl:k=4} through \ref{tbl:k=6} corresponding to higher uniformities with the exception of this first column. The only instance of a primitive uniquely $K_4^{(3)}$-saturated hypergraph encountered throughout our search exists on $n=7$ vertices and corresponds to the Cs\'{a}sz\'{a}r polyhedron \cite{c-49} when each triangular face is considered as an edge. Otherwise, it is unknown whether or not any other such hypergraph exists, motivating the following question.

\begin{question}
  Are there only finitely many primitive uniquely $K_4^{(3)}$-saturated hypergraphs? 
\end{question}

Since primitive uniquely $K_3$-saturated graphs were shown to be the Moore graphs of diameter two in $\cite{Cooper}$, primitive uniquely $K_4^{(3)}$-saturated hypergraphs may be viewed as a $3$-uniform generalization of these Moore graphs, of which there are only finitely many \cite{Hoffman}.

Furthermore, with the exception of this first column of Table~\ref{tbl:k=3}, it appears that each row in Tables~\ref{tbl:k=3} through \ref{tbl:k=6} is ``monotone" in the sense the hypergraphs always exist in the first portion of the row, including those parameter combinations falling in between the ranges given by Theorems~\ref{thm:double_star} and \ref{thm:tau_cosntruction_hypergraphs}, but once a parameter combination is encountered for which no hypergraph exists, no such hypergraph exists for larger values of $r$ and $n$. This suggests a possible generalization of Theorem~\ref{thm:n-r=1} that holds for each row of the tables. 

\begin{question}
  Given integers $k \geq 3,\ell \geq 1$, can we completely determine a range for $n$ where primitive uniquely $K^{(k)}_{n - \ell}$-saturated hypergraphs exist on $n$ vertices?
\end{question}

Such a result would completely characterize the values of $k,r,n$ for which primitive uniquely $K^{(k)}_{r}$-saturated hypergraphs exist on $n$ vertices.

\bibliographystyle{abbrv}
\bibliography{mybib}

\begin{thebibliography}{10}

\bibitem{Berman}
L.~W. Berman, G.~G. Chappell, J.~R. Faudree, J.~Gimbel, and C.~Hartman.
\newblock Uniquely tree-saturated graphs.
\newblock {\em Graphs Combin.}, 32(2):463--494, 2016.

\bibitem{Cooper}
J.~Cooper, J.~Lenz, T.~D. LeSaulnier, P.~S. Wenger, and D.~B. West.
\newblock Uniquely {$C_4$}-saturated graphs.
\newblock {\em Graphs Combin.}, 28(2):189--197, 2012.

\bibitem{c-49}
A.~Cs\'asz\'ar.
\newblock A polyhedron without diagonals.
\newblock {\em Acta Univ. Szeged. Sect. Sci. Math.}, 13:140--142, 1949.

\bibitem{Gallai}
P.~Erd\H{o}s and T.~Gallai.
\newblock On the minimal number of vertices representing the edges of a graph.
\newblock {\em Magyar Tud. Akad. Mat. Kutat\'o Int. K\"ozl.}, 6:181--203, 1961.

\bibitem{Moon}
P.~Erd\H{o}s, A.~Hajnal, and J.~W. Moon.
\newblock A problem in graph theory.
\newblock {\em Amer. Math. Monthly}, 71:1107--1110, 1964.

\bibitem{eb-96}
T.~Etzion and S.~Bitan.
\newblock On the chromatic number, colorings, and codes of the {J}ohnson graph.
\newblock {\em Discrete Appl. Math.}, 70(2):163--175, 1996.

\bibitem{gs-80}
R.~L. Graham and N.~J.~A. Sloane.
\newblock Lower bounds for constant weight codes.
\newblock {\em IEEE Trans. Inform. Theory}, 26(1):37--43, 1980.

\bibitem{Hartke}
S.~G. Hartke and D.~Stolee.
\newblock Uniquely {$K_r$}-saturated graphs.
\newblock {\em Electron. J. Combin.}, 19(4):Paper 6, 39, 2012.

\bibitem{Hoffman}
A.~J. Hoffman and R.~R. Singleton.
\newblock On {M}oore graphs with diameters {$2$} and {$3$}.
\newblock {\em IBM J. Res. Develop.}, 4:497--504, 1960.

\bibitem{Turan}
P.~Tur\'an.
\newblock Eine {E}xtremalaufgabe aus der {G}raphentheorie.
\newblock {\em Mat. Fiz. Lapok}, 48:436--452, 1941.

\bibitem{Tuza}
Z.~Tuza.
\newblock Minimum number of elements representing a set system of given rank.
\newblock {\em J. Combin. Theory Ser. A}, 52(1):84--89, 1989.

\bibitem{Wenger}
P.~S. Wenger and D.~B. West.
\newblock Uniquely cycle-saturated graphs.
\newblock {\em J. Graph Theory}, 85(1):94--106, 2017.

\end{thebibliography}

\end{document}